\documentclass[smallextended,
    envcountsect,
    envcountsame,
    envcountreset]{svjour3}
    
\usepackage{graphicx}
\usepackage{booktabs}
\usepackage{amsmath}
\usepackage{amstext}
\usepackage{amsfonts}
\usepackage{amssymb}
\usepackage{amsxtra} 
\usepackage{algorithm}
\usepackage{algorithmic}
\usepackage{hyperref}
\usepackage{bbm}
\usepackage{url}
\usepackage[dvipsnames]{xcolor}
\usepackage{tikz}
\usepackage{float} 
\usetikzlibrary{external}
\usepackage{pgfplots}
\usepackage{pgfplotstable}
\pgfplotsset{compat=newest}
\pgfplotsset{plot coordinates/math parser=false}
\newlength\figureheight
\newlength\figurewidth
\usepackage{xargs} 
\usepackage{todonotes}
\newcommandx{\DLtodo}[2][1=]{\todo[color=green!40,#1]{DL: #2}}
\newcommandx{\FStodo}[2][1=]{\todo[color=orange!40,#1]{FS: #2}}
\newcommandx{\LTtodo}[2][1=]{\todo[color=blue!40,#1]{LT: #2}}
\newcommandx{\MWtodo}[2][1=]{\todo[color=red!40,#1]{MW: #2}}

\newcolumntype{M}[1]{>{\centering\arraybackslash}m{#1}}
\usepackage[square,numbers,sort&compress]{natbib}

\graphicspath{ {./figures/} }

\smartqed

\newcommand \NN {\mathbb{N}}
\newcommand \RR {\mathbb{R}}
\newcommand \CC {\mathbb{C}}
\newcommand \EE {\mathbb{E}}

\newcommand \KK {\mathbb{K}}

\newcommand \Rcal {\mathcal{R}}

\newcommand \Ncal {\mathcal{N}}

\DeclareMathOperator*{\argmin}{argmin}
\DeclareMathOperator{\sign}{sign}

\DeclareMathOperator{\rint}{rint}

\DeclareMathOperator{\dist}{dist}

\newcommand{\scp}[2]{\langle #1 \,,\, #2\rangle}
\newcommand{\norm}[2][]{\|#2\|_{#1}}
\newcommand{\set}[2]{\{#1\, |\, #2\}}

\date{}
\journalname{}

\begin{document}

\title{Extended Randomized Kaczmarz Method for Sparse Least Squares and Impulsive Noise Problems\thanks{The work of L.T. and D.L. has been supported by the ITN-ETN project TraDE-OPT funded by the European Union’s Horizon 2020 research and innovation programme under the Marie Skłodowska-Curie grant agreement No 861137. This work represents only the author’s view and the European Commission is not responsible for any use that may be made of the information it contains.}}
\titlerunning{Generalized extended Randomized Kaczmarz Method}

\author{Frank~Sch\"{o}pfer \and Dirk~A. Lorenz \and Lionel~Tondji \and Maximilian~Winkler}
\institute{Frank Sch\"{o}pfer
  \at Institut f\"ur Mathematik, Carl von Ossietzky Universit\"at Oldenburg, 26111 Oldenburg, Germany,\\
  \email{frank.schoepfer@uni-oldenburg.de}
  \and Dirk A. Lorenz
  \at Institute for Analysis and Algebra, TU Braunschweig, 38092 Braunschweig, Germany,\\
  \email{d.lorenz@tu-braunschweig.de}
  \and Lionel Tondji
  \at Institute for Analysis and Algebra, TU Braunschweig, 38092 Braunschweig, Germany,\\
  \email{l.ngoupeyou-tondji@tu-braunschweig.de}
  \and Maximilian Winkler
  \at Institute for Analysis and Algebra, TU Braunschweig, 38092 Braunschweig, Germany,\\
  \email{maximilian.winkler@tu-braunschweig.de}
  }

\maketitle

\begin{abstract}
The Extended Randomized Kaczmarz method is a well known iterative scheme which can find the Moore-Penrose inverse solution of a possibly inconsistent linear system and requires only one additional column of the system matrix in each iteration in comparison with the standard randomized Kaczmarz method. Also, the Sparse Randomized Kaczmarz method has been shown to converge linearly to a sparse solution of a consistent linear system. Here, we combine both ideas and propose an Extended Sparse Randomized Kaczmarz method. We show linear expected convergence to a sparse least squares solution in the sense that an extended variant of the regularized basis pursuit problem is solved. Moreover, we generalize the additional step in the method and prove convergence to a more abstract optimization problem. We demonstrate numerically that our method can find sparse least squares solutions of real and complex systems if the noise is concentrated in the complement of the range of the system matrix and that our generalization can handle impulsive noise.
\end{abstract}



\section{Introduction}

We consider the fundamental problem of approximating sparse solutions of large and possibly inconsistent linear systems
\[
Ax = b
\]
with matrix $A \in \KK^{m \times n}$ and right hand side $b \in \KK^{m}$, in the real case $\KK=\RR$ as well as in the complex case $\KK=\CC$.
In particular, we have in mind situations where $A=M \cdot D$ is the product of a tall matrix $M \in \KK^{m \times r}$ with $m > r$, and a matrix $D \in \KK^{r \times n}$ with $r \le n$, which acts as a basis or overcomplete dictionary that allows for a sparse representation of the solution, and where the given data $b$ may be corrupted by noise and need not be contained in the range $\Rcal(A)$ of $A$.
This setting is somewhat more general than the usual one in the field of compressed sensing~\cite{Don06_CS}, where mostly flat matrices $A$ with $m << n$ and full row rank are considered.
It arises e.g. in geophysical sparsity-promoting imaging problems~\cite{YWFH16}, where the system matrix is the product of a Curvelet transform matrix, which is suitable for a sparse representation of the solution, and a Jacobian, which corresponds to a linearized Born model, so that besides noisy measurement data there is also inconsistency due to a linearization error.

Here we set out to tackle such problems by solving combined optimization problems of the form
\begin{gather}
\min_{x \in \KK^n} f(x) \quad \mbox{s.t.}\quad Ax=\hat{y},\\ \nonumber 
\mbox{where}\quad \hat{y} = \argmin_{y \in \KK^m} g^*(b-y) \quad \mbox{s.t.}\quad y \in \Rcal(A)\label{eq:OPgeneral}
\end{gather}
with sparsity promoting functions $f$ and suitable data misfit functions $g^*$.
For instance, it is known that the choice $f(x)=\lambda \cdot \norm[1]{x} + \tfrac{1}{2} \cdot \norm[2]{x}^{2}$ favors sparse solutions for appropriate choices of $\lambda>0$, see~\cite{Don06,COS09,Sch12}, where $\norm[1]{x}$ and $\norm[2]{x}$ denote the $\ell_1$-norm and $\ell_2$-norm of $x$, respectively.
Similarly, by dividing the components of $x$ into $K$ groups $x=(x_1,\ldots,x_K)$ with $x_j \in \KK^{n_j}$, the function $f(x) = \lambda \cdot \sum_{j=1}^K \norm[2]{x_j} + \tfrac{1}{2} \cdot \norm[2]{x}^{2}$ favors group sparsity~\cite{SPH09}.
And in the related area of low rank matrix solutions~\cite{CCS08,RFP10} we may choose $f(X)=\lambda  \cdot \norm[*]{X} + \tfrac{1}{2} \cdot \norm[F]{X}^2$, where $\norm[*]{X}$ and $\norm[F]{X}$ denote the nuclear norm and Frobenius norm of a matrix $X$, respectively.
Suitable data misfit functions are $g^*(b-y)=\tfrac{1}{2} \cdot \norm[2]{b-y}^{2}$ for least squares solutions, and $\ell_1$-norm-like functions in situations where the data $b$ is corrupted by impulsive noise, i.e. the case where only some components of the data are faulty, but with possibly large errors, see~\cite{YZ11,SKPB11,WLLQY16}.

The linear system may be so large that full matrix operations are very expensive or even infeasible.
Then it appears desirable to use iterative algorithms with low computational cost and storage per iteration that produce good approximate solutions of~\eqref{eq:OPgeneral} after relatively few iterations.
A celebrated example for the computation of minimum $\ell_2$-norm solutions of consistent linear systems is the Kaczmarz method~\cite{Kac37}, also known as Algebraic Reconstruction Technique (ART), and its block and randomized variants~\cite{NT14} which started to get popular due to the seminal paper~\cite{SV09}.
In its most simple form for $\KK=\RR$, in each iteration a row vector $a_i^T$ of $A$ is chosen at random and the new iterate $x_{k+1}$ is then computed as the orthogonal projection of $x_k$ onto the solution hyperplane corresponding to the $i$-th equation $\scp{a_i}{x}=b_i$, i.e.\footnote{We use subscript indices for components of a vector, columns or rows of a matrix, and also as iteration indices. But the meaning should always be clear from the context.}
\[
x_{k+1} = x_k - \tfrac{\scp{a_i}{x_k}-b_i}{\norm[2]{a_i}^2} \cdot a_i \,,
\]
with initial value $x_0=0$.
The Randomized Sparse Kaczmarz method~\cite{SL19,LWSM14,P15} is a relatively new variant of the Kaczmarz method with almost the same low cost and storage requirements, and which has shown good performance in approximating sparse solutions of large consistent linear systems.
It uses two variables $x_k^*$ and $x_k$ and reads as
\begin{align*}
x_{k+1}^* &= x_k^* - \tfrac{\scp{a_i}{x_k}-b_i}{\norm[2]{a_i}^2} \cdot a_i \,,\\
x_{k+1} &= S_{\lambda}(x_{k+1}^*)
\end{align*}
with initial values $x_0=x_0^*=0$, and the soft shrinkage operator, which acts componentwise on a vector $x$ as
\begin{equation} \label{eq:S}
\big(S_{\lambda}(x))_j = \max\{|x_j|-\lambda,0\} \cdot \sign(x_j)\,.
\end{equation}
We refer the interested reader to~\cite{CQ21} for an extension of this algorithm to sparse tensor recovery problems.

For consistent systems the iterates of the Randomized Sparse Kaczmarz method converge in expectation to the solution of the regularized \emph{Basis Pursuit Problem}
\[
\min_{x \in \RR^n} \lambda \cdot \norm[1]{x} + \tfrac{1}{2} \cdot \norm[2]{x}^{2} \quad \mbox{s.t.}\quad Ax=b \,.
\]
However, for inconsistent systems the iterates do not converge, see~\cite{DHK20,SL19} for a detailed study of this phenomenon.
This behaviour is also well-known for the vanilla Kaczmarz method.
As a remedy, in~\cite{ZF13,Du19} an Extended Randomized Kaczmarz method was proposed, which additionally uses one column $\tilde{a}_j$ of $A$ in each step and finds the Moore-Penrose inverse solution, i.e. the least squares solution with minimum $\ell^2$-norm.
Using an additional variable $z_k$ with initial value $z_0=b$, the iterates are computed as
\begin{align*}
z_{k+1} &= z_k - \tfrac{\scp{\tilde{a}_j}{z_k}}{\norm[2]{\tilde{a}_j}^2} \cdot \tilde{a}_j \,,\\
x_{k+1} &= x_k - \tfrac{\scp{a_i}{x_k}-b_i+z_{k+1,i}}{\norm[2]{a_i}^2} \cdot a_i \,.
\end{align*}
Advantages of using block variants have recently been discussed in~\cite{DSS20,Wu22}, and generalizations to tensor recovery problems are currently in preparation~\cite{DS21}.

In this paper, we adopt these ideas and propose the Generalized Extended Randomized Kaczmarz method to solve~\eqref{eq:OPgeneral}, see Algorithm~\ref{alg:GERK}.
For example, to obtain sparse least squares solutions via
\begin{gather*}
\min_{x \in \RR^n} \lambda \cdot \norm[1]{x} + \tfrac{1}{2} \cdot \norm[2]{x}^{2} \quad \mbox{s.t.}\quad Ax=\hat{y},\\ \nonumber 
\mbox{where}\quad \hat{y} = \argmin_{y \in \RR^m} \tfrac{1}{2} \cdot \norm[2]{b-y}^{2} \quad \mbox{s.t.}\quad y \in \Rcal(A)
\end{gather*}
the iteration reads as
\begin{align*}
z_{k+1} &= z_k - \tfrac{\scp{\tilde{a}_j}{z_k}}{\norm[2]{\tilde{a}_j}^2} \cdot \tilde{a}_j \,,\\
x_{k+1}^* &= x_k^* - \tfrac{\scp{a_i}{x_k}-b_i+z_{k+1,i}}{\norm[2]{a_i}^2} \cdot a_i \,,\\
x_{k+1} &= S_{\lambda}(x_{k+1}^*) \,,
\end{align*}
where $\tilde{a}_j$ is the $j$-th column of $A$.
We prove expected convergence with rates under appropriate assumptions for general functions $f$ and $g^*$ with the help of global error bounds.
We also consider block versions, and 
in particular, convergence in the complex case $\KK=\CC$ is shown by considering the iteration as a suitable block method in real variables.

In the next section we recall some basic notions and properties of convex functions and Bregman distances, which will be used to analyze the iteration methods in Section~\ref{sec:convergence}.
Convergence rates will be derived with the help of global error bounds from Section~\ref{sec:errorbounds}.
The theoretical results are supported by numerical examples for sparse solutions of real and complex inconsistent systems under different noise models in Section~\ref{sec:NumericalExamples}.

\section{Preliminaries}

For $x,y\in\RR^n$, we denote the standard inner product by $\scp{x}{y}$ and for $p\in [1,+\infty[$ the $\ell_p$ norm by
$$\|x\|_p := \big(\sum_{i=1}^n |x_i|^p\big)^{\tfrac1p}.$$ For a nonempty closed convex set $C \subset \RR^n$, we write its Euclidean projector as $P_C$ and its distance function by 
$$ \mathrm{dist}(x,C) := \inf_{z\in C} \|x-z\|_2. $$ 

As in~\cite{SL19} we will analyze the convergence of the algorithms with the help of the Bregman distance~\cite{Bre67} with respect to the objective function $f$.
To this end we recall some well known concepts and properties of convex functions~\cite{RW09}.
Let $f:\RR^n \to \RR$ be convex and finite everywhere.
Then $f$ is continuous and its \emph{subdifferential}
\[
\partial f(x) := \set{x^* \in \RR^n}{ f(y) \ge f(x) + \scp{x^*}{y-x}\: \mbox{for all $y \in \RR^n$}}
\]
at any $x \in \RR^n$ is nonempty, compact and convex.
Throughout the paper we assume that $f$ is even \emph{strongly convex}, i.e. there is some $\alpha>0$ such that for all $x,y \in \RR^n$ and \emph{subgradients} $x^* \in \partial f(x)$ we have
\[
f(y) \ge f(x) + \scp{x^*}{y-x} + \tfrac{\alpha}{2} \cdot \norm[2]{y-x}^2 \,.
\]
Then $f$ is \emph{coercive}, i.e.
\[
\lim_{\norm[2]{x} \to \infty} f(x)=\infty \,,
\]
and its \emph{conjugate function} $f^*:\RR^n \to \RR$ with
\[
f^*(x^*):= \sup_{y \in \RR^n} \scp{x^*}{y} - f(y)
\]
is also convex, finite everywhere and coercive.
Additionally, $f^*$ is differentiable with a \emph{Lipschitz-continuous gradient} with constant $L_{f^*}=\frac{1}{\alpha}$, i.e. for all $x^*,y^* \in \RR^n$ we have
\[
\norm[2]{\nabla f^*(x^*)-\nabla f^*(y^*)} \le L_{f^*} \cdot \norm[2]{x^*-y^*} \,,
\]
which implies the estimate
\begin{equation} \label{eq:Lip}
f^*(y^*) \le f^*(x^*) -\scp{\nabla f^*(x^*)}{y^*-x^*} + \tfrac{L_{f^*}}{2} \cdot \norm[2]{x^*-y^*}^2 \,.
\end{equation}

\begin{example}[cf.~\cite{LSW14,Yin10}] \label{exmp:f}
The sparsity promoting objective function
\begin{equation} \label{eq:spf}
f(x) := \lambda \cdot \norm[1]{x} + \tfrac{1}{2} \cdot \norm[2]{x}^{2}
\end{equation}
is strongly convex with constant $\alpha=1$ for any $\lambda\geq 0$, its subdifferential is
\[
\partial f(x)=\set{x+\lambda \cdot s}{\mbox{$s_j=\sign(x_j)$ if $x_j\not=0$, and $s_j\in [-1,1]$ if $x_j=0$}} \,,
\]
and its conjugate function can be computed with the soft shrinkage operator~\eqref{eq:S} as
\[
f^{*}(x^{*}) = \tfrac{1}{2} \cdot \norm[2]{S_{\lambda}(x^{*})}^{2} \quad \mbox{with} \quad \nabla f^{*}(x^{*}) = S_{\lambda}(x^{*}) \,.
\]
\end{example}

\begin{definition} \label{def:D}
The \emph{Bregman distance} $D_f^{x^*}(x,y)$ between $x,y \in \RR^n$ with respect to $f$ and a subgradient $x^* \in \partial f(x)$ is defined as
\[
D_f^{x^*}(x,y):=f(y)-f(x) -\scp{x^*}{y - x}\,.
\]
\end{definition}
Fenchel's equality states that $f(x) + f^*(x^*) = \scp{x}{x^*}$ if $x^*\in\partial f(x)$ and implies that the Bregman distance can be written as
\[
D_f^{x^*}(x,y) = f^*(x^*)-\scp{x^*}{y} + f(y)\,.
\]

\begin{example}[cf.~\cite{SL19}]  \label{exmp:D}
For $f(x)=\frac{1}{2} \cdot \norm[2]{x}^2$ we just have $D_f^{x^*}(x,y)=\frac{1}{2}\norm[2]{x-y}^2$.
For $f(x) = \lambda \cdot \norm[1]{x} + \tfrac{1}{2} \cdot \norm[2]{x}^{2}$ and any $x^*=x+\lambda \cdot s \in \partial f(x)$ we have
\[
D_f^{x^*}(x,y)=\frac{1}{2} \cdot \norm[2]{x-y}^2 + \lambda \cdot(\norm[1]{y}-\scp{s}{y}) \,.
\]
\end{example}

The following inequalities are crucial for the convergence analysis of the randomized algorithms.
They immediately follow from the definition of the Bregman distance and the assumption of strong convexity of $f$, cf.~\cite{LSW14}.
For all $x,y \in \RR^n$ and $x^* \in \partial f(x)$, $y^* \in \partial f(y)$ we have
\begin{equation} \label{eq:D}
\frac{\alpha}{2} \norm[2]{x-y}^2 \le  D_f^{x^*}(x,y) \le \scp{x^*-y^*}{x-y} \le \norm[2]{x^*-y^*} \cdot \norm[2]{x-y}
\end{equation}

Note that if $f$ is differentiable with a Lipschitz-continuous gradient, then we also have the (better) upper estimate $D_f^{x^*}(x,y) \le L_f \cdot \norm[2]{x-y}^2$, but in general this need not be the case.
The following example was also used in~\cite{P15} as a smoothed version of~\eqref{eq:spf}.

\begin{example}  \label{exmp:Huber}
For $\varepsilon>0$ the \emph{Huber function}~\cite{Hub73} is defined by
\[
r_{\varepsilon}(x) := \sum_{j=1}^n \begin{cases}
    |x_j|-\tfrac{\varepsilon}{2} &, |x_j|>\varepsilon \\
    \tfrac{1}{2 \cdot \varepsilon} \cdot x_j^2 &, |x_j|\le \varepsilon.
\end{cases} \,.
\]
Then for $\tau>0$ the function
\[
f(x) := r_{\varepsilon}(x) + \tfrac{\tau}{2} \cdot \norm[2]{x}^2, \quad 
\]
is $\tau$-strongly convex and has a $\big(\tfrac{1}{\varepsilon}+\tau)$-Lipschitz-continuous gradient with
\[
\big(\nabla f(x)\big)_j = \big(\tfrac{1}{\max(\varepsilon,|x_j|)}+\tau\big) \cdot x_j \,.
\]
\end{example}

\section{Error bounds for linearly constrained optimization problems}
\label{sec:errorbounds}

Consider the feasible, convex and linearly constrained optimization problem
\begin{equation}
\min_{x \in \RR^n} f(x) \quad \mbox{s.t.} \quad Ax=b \label{eq:OPfeas}
\end{equation}
with a nonzero matrix $A \in \RR^{m \times n}$, right hand side $b \in \Rcal(A)$, and strongly convex objective function $f:\RR^n \to \RR$.
This problem has a unique solution $\hat{x}$ which fulfills $\partial f(\hat{x}) \cap \Rcal(A^T) \not=\emptyset$.
To obtain convergence rates for the solution algorithms, we will estimate the Bregman distance of the iterates to the solution $\hat{x}$ by \emph{error bounds} of the form $D_f^{x^*}(x,\hat{x}) \le \gamma \cdot \norm[2]{Ax-b}$ or $D_f^{x^*}(x,\hat{x}) \le \gamma \cdot \norm[2]{Ax-b}^2$.
We will see that such error bounds always hold if $f$ has a Lipschitz-continuous gradient.
But they also hold under weaker conditions.
The following example was already proved in~\cite{SL19} (and here it also follows from Theorem~\ref{thm:EB} below).

\begin{example} \label{exmp:EBsparse}
Let $\hat{x}$ be the unique solution of~\eqref{eq:OPfeas} with objective function $f(x) = \lambda \cdot \norm[1]{x} + \tfrac{1}{2} \cdot \norm[2]{x}^{2}$.
Then there exists $\gamma(\hat{x}) >0$ such that for all $x \in \RR^n$ and $x^* \in \partial f(x) \cap \Rcal(A^T)$ we have
\[
D_f^{x^*}(x,\hat{x})\le \gamma(\hat{x}) \cdot \norm[2]{Ax-b}^2\,.
\]
Based on the results of~\cite{LY13}, an explicit expression of $\gamma(\hat{x})$ for $\hat{x} \not= 0$ was given in~\cite{SL19} as follows:
Let $A_J\not=0$ denote a submatrix that is formed by the columns of $A$ indexed by $J \subset \{1,\ldots,n\}$, and let $\sigma_{\min}^{+}(A_J)$ denote its smallest positive singular value.
We set
\[
\tilde{\sigma}_{\min}(A):=\min\set{\sigma_{\min}^{+}(A_J)}{J \subset \{1,\ldots,n\}, A_J \not=0} \,,
\]
and for $\hat{x} \not= 0$ we define $|\hat{x}|_{\min}=\min\set{|\hat{x}_j|}{\hat{x}_j \not=0}$.
Then we have
\[
\gamma(\hat{x})=\frac{1}{\tilde{\sigma}_{\min}^2(A)} \cdot \frac{|\hat{x}|_{\min} + 2 \lambda}{|\hat{x}|_{\min}} \,.
\]
Moreover, for $\hat{x} = 0$ we may use $\gamma(0)=\frac{2n}{\big(\sigma_{\min}^{+}(A)\big)^2}$ (this can be shown with inequality~\eqref{eq:D_estimate} in the beginning of the proof of Lemma~\ref{lem:lineargrowth} below, but since this explicit expression is not so important here, we omit the details).
Note that $\gamma(\hat{x})$ is quite discontinuous with respect to $\hat{x}$ and may become arbitrarily large, since $\displaystyle \lim_{\hat{x}\not=0, |\hat{x}|_{\min} \to 0} \gamma(\hat{x})=\infty$.
We do not know whether these expressions for $\gamma(\hat{x})$ are the best possible.
\end{example}

To clarify the assumptions under which such error bounds hold for more general objective functions, we introduce the concepts of calmness~\cite{RW09} and linear regularity~\cite{BBL99}.
Let $B_2$ denote the closed unit ball of the $\ell_2$-norm.

\begin{definition}
The (set-valued) subdifferential mapping $\partial f:\RR^n \rightrightarrows \RR^n$ is \emph{calm} at $\hat{x}$ if there are constants $\varepsilon,L>0$ such that
\begin{equation}
\partial f(x) \subset \partial f(\hat{x}) + L \cdot \norm[2]{x-\hat{x}} \cdot B_2 \quad \mbox{for any $x$ with} \quad \norm[2]{x-\hat{x}}\le \varepsilon\,. \label{eq:calm}
\end{equation} 
\end{definition}

Note that calmness is a local growth condition similar to Lipschitz-continuity of a gradient mapping, but for fixed $\hat{x}$.
Furthermore, this does not imply that for all $x^* \in \partial f(x)$ and all $\hat{x}^* \in \partial f(\hat{x})$ we have $\norm[2]{x^*-\hat{x}^*} \le L \cdot \norm[2]{x-\hat{x}}$, but only for some $\hat{x}^*$ which may depend on $x^*$.
Of course, any Lipschitz-continuous gradient mapping is calm everywhere.

\begin{example} \label{exmp:calm}
\begin{enumerate} 
\item The subdifferential mapping of any convex piecewise linear-quadratic function $f:\RR^n \to \RR$ is calm everywhere. In particular, this holds for $f(x) = \lambda \cdot \norm[1]{x} + \tfrac{1}{2} \cdot \norm[2]{x}^{2}$.
\item For matrices $X \in \RR^{n_1 \times n_2}$ the subdifferential mapping of $f(X)=\lambda  \cdot \norm[*]{X} + \tfrac{1}{2} \cdot \norm[F]{X}^2$ is calm everywhere.
\item The subdifferential mapping of
\[
f(x)=\lambda \cdot \norm[2]{x} + \tfrac{1}{2} \cdot \norm[2]{x}^2
\]
is calm everywhere with
\[
\partial f(x)=
\begin{cases}
\lambda \cdot \tfrac{x}{\norm[2]{x}} + x, & x \not= 0 \\
\lambda \cdot B_2, & x=0
\end{cases}
\]
and it holds that 
\[
  \begin{split}
    f^*(x^*)&=\tfrac{1}{2} \cdot \norm[2]{x^*-P_{\lambda \cdot  B_2}(x^*)}^2,\\
    \nabla f^*(x^*)&=x^*-P_{\lambda \cdot B_2}(x^*) = \max\left\{0,1-\tfrac{\lambda}{\norm[2]{x^*}}\right\} \cdot x^*.
  \end{split}
\]
\item Divide the components of $x$ into $K$ groups $x=(x_1,\ldots,x_K)$ with $x_j \in \RR^{n_j}$.
Then the subdifferential mapping of $f(x) = \lambda \cdot \sum_{j=1}^K \norm[2]{x_j} + \tfrac{1}{2} \cdot \norm[2]{x}^{2}$ is calm everywhere.
\end{enumerate}

The cases (a) and (b) were already proven in~\cite{Sch16}, and (d) is the group-version of (c),
hence, we show (c). For $\hat{x}\not=0$ and $x \not=0$ the function $f$ is indeed differentiable with
\[
\norm[2]{\nabla f(x) - \nabla f(\hat{x})} \le \left(1+\tfrac{2 \lambda}{\norm[2]{\hat{x}}} \right) \cdot \norm[2]{x-\hat{x}} \,,
\]
i.e.~\eqref{eq:calm} holds with $L=1+\tfrac{2\lambda}{\norm[2]{\hat{x}}}$ for all $x$ with $\norm[2]{x-\hat{x}} < \frac{\norm[2]{\hat{x}}}{2}$.
For $\hat{x}=0$ and $x \not=0$ we have $\nabla f(x)=\lambda \cdot \tfrac{x}{\norm[2]{x}} + x \in \partial f(\hat{x}) + \norm[2]{x-\hat{x}} \cdot \tfrac{x-\hat{x}}{\norm[2]{x-\hat{x}}}$, i.e.~\eqref{eq:calm} holds with $L=1$ for all $x$ since $\tfrac{x-\hat{x}}{\norm[2]{x-\hat{x}}} \in B_2$.
\qed
\end{example}

\begin{definition}
Let $\partial f(x) \cap \Rcal(A^T)\not=\emptyset$.
Then the collection $\{\partial f(\hat{x}), \Rcal(A^T)\}$ is \emph{linearly regular}, if there is a constant $\gamma>0$ such that for all $x^* \in \RR^n$ we have 
\begin{equation}
\dist\big(x^*,\partial f(\hat{x}) \cap \Rcal(A^T)\big) \le \gamma \cdot \Big( \dist\big(x^*,\partial f(\hat{x}) \big) +  \dist\big(x^*,\Rcal(A^T)\big) \Big)\,. \label{eq:linreg}
\end{equation}
\end{definition}

Obviously, if $f$ is differentiable at $\hat{x}$, i.e. if $\partial f(\hat{x})=\{\nabla f(\hat{x})\}$ is a singleton, then we have linear regularity.

\begin{example}[cf.~\cite{BBL99,SL19}] \label{exmp:linreg}
The collection $\{\partial f(\hat{x}), \Rcal(A^T)\}$ is linearly regular, if
\begin{enumerate}
\item $\partial f(\hat{x})$ is polyhedral (which holds for piecewise linear-quadratic $f$ in particular), or if
\item $\rint\big(\partial f(\hat{x})\big) \cap \Rcal(A^T)\not=\emptyset$, where $\rint\big(\partial f(\hat{x})\big)$ denotes the relative interior of $\partial f(\hat{x})$.
\end{enumerate}
\end{example}

The condition in Example~\ref{exmp:linreg}~(b) is a standard regularity assumption, similar to the Slater condition. 
In~\cite{SL19} local error bounds were sufficient to prove convergence, because all iterates were guaranteed to be bounded.
In the present paper this need not be the case (we will in general only show boundedness in expectation).
But here we will derive global error bounds under a global growth condition on the subdifferential mapping of $f$.

\begin{definition}
We say the subdifferential mapping of $f$ \emph{grows at most linearly}, if there exist $\eta,\rho \ge 0$ such that for all $x \in \RR^n$ and $x^* \in \partial f(x)$ we have
\begin{equation} \label{eq:lineargrowth}
\norm[2]{x^*} \le \eta \cdot \norm[2]{x} + \rho\,.
\end{equation}
\end{definition}

\begin{example} \label{exmp:LinearGrowth}
Any Lipschitz-continuous gradient mapping grows at most linearly.
Furthermore, the subdifferential mappings of all functions in Example~\ref{exmp:calm} grow at most linearly.
\end{example}

\begin{lemma} \label{lem:lineargrowth}
Let $\hat{x}$ be the unique solution of~\eqref{eq:OPfeas}.
If the subdifferential mapping of $f$ grows at most linearly, then there exists some constant $c>0$ such that for all $x \in \RR^n$ and $x^* \in \partial f(x) \cap \Rcal(A^T)$ we have
\[
D_f^{x^*}(x,\hat{x}) \le c \cdot \left( \sqrt{D_f^{x^*}(x,\hat{x})} + \norm[2]{\hat{x}} +1 \right) \cdot \norm[2]{Ax-b} \,.
\]
\end{lemma}

\begin{proof}
To $\hat{x}$ there is some $\hat{x}^* \in \partial f(\hat{x}) \cap \Rcal(A^T)$.
We choose $u,\hat{u} \in \Ncal(A^T)^\bot$ with $x^*=A^T u$ and $\hat{x}^*=A^T \hat{u}$, and by~\eqref{eq:D} we estimate
\begin{align} \label{eq:D_estimate}
D_f^{x^*}(x,\hat{x}) & \le \scp{x^*-\hat{x}^*}{x-\hat{x}} = \scp{A^Tu-A^T\hat{u}}{x-\hat{x}} = \scp{u-\hat{u}}{Ax-b} \nonumber \\
& \le \norm[2]{u-\hat{u}}\cdot \norm[2]{Ax-b} \le \tfrac{1}{\sigma_{\min}^+(A)} \cdot \norm[2]{A^Tu-A^T\hat{u}} \cdot \norm[2]{Ax-b} \nonumber \\
& = \tfrac{1}{\sigma_{\min}^+(A)} \cdot \norm[2]{x^*-\hat{x}^*} \cdot \norm[2]{Ax-b} \,.
\end{align}
It remains to estimate $\norm[2]{x^*-\hat{x}^*}$.
The assumption of at most linear growth~\eqref{eq:lineargrowth} together with~\eqref{eq:D} implies
\begin{align*}
\norm[2]{x^*-\hat{x}^*} &\le \eta \cdot (\norm[2]{x} + \norm[2]{\hat{x}}) + 2 \rho \\
&\le \eta \cdot \norm[2]{x-\hat{x}} + 2 \cdot (\eta \cdot \norm[2]{\hat{x}} +\rho) \\
& \le \eta \cdot \sqrt{\tfrac{2}{\alpha} \cdot D_f^{x^*}(x,\hat{x})}   + 2 \cdot (\eta \cdot \norm[2]{\hat{x}} +\rho) \,,
\end{align*}
from which the assertion follows.
\qed
\end{proof}

Now we can derive the global error bound.

\begin{theorem} \label{thm:EB}
Let $f:\RR^n\to \RR$ be strongly convex.
If its subdifferential mapping grows at most linearly, is calm at the unique solution $\hat{x}$ of~\eqref{eq:OPfeas}, and if the collection $\{\partial f(\hat{x}), \Rcal(A^T)\}$ is linearly regular, then there exists $\gamma(\hat{x})>0$ such that for all $x \in \RR^n$ and $x^* \in \partial f(x) \cap \Rcal(A^T)$ we have the global error bound
\begin{equation}
D_f^{x^*}(x,\hat{x})\le \gamma(\hat{x}) \cdot \norm[2]{Ax-b}^2\,. \label{eq:EB}
\end{equation}
In particular, this holds if $f$ has a Lipschitz-continuous gradient.
\end{theorem}

\begin{proof}
Let $\alpha>0$ be the strong convexity constant, and let $\varepsilon, L>0$ be as in~\eqref{eq:calm} in the definition of calmness.
At first we consider the case $D_f^{x^*}(x,\hat{x}) \le \tfrac{\alpha}{2} \cdot \varepsilon^2$.
Then by~\eqref{eq:D} we have $\norm[2]{x-\hat{x}} \le \varepsilon$, so that by~\eqref{eq:calm} and~\eqref{eq:D_estimate} we get
\begin{equation} \label{eq:DL}
    \dist\big(x^*,\partial f(\hat{x})\big) \le L \cdot \norm[2]{x-\hat{x}} \le L \cdot \sqrt{\tfrac{2}{\alpha} \cdot D_f^{x^*}(x,\hat{x})} \,.
\end{equation}
Let $C:=\partial f(\hat{x}) \cap \Rcal(A^T)$.
By choosing $\hat{x}^*:=P_C(x^*)$ in~\eqref{eq:D_estimate} in the beginning of the proof of Lemma~\ref{lem:lineargrowth}, we conclude that
\[
D_f^{x^*}(x,\hat{x}) \le \tfrac{1}{\sigma_{\min}^+(A)} \cdot \dist\big(x^*,C) \cdot \norm[2]{Ax-b} \,.
\]
Since $x^* \in \Rcal(A^T)$, linear regularity~\eqref{eq:linreg} ensures that \[
\dist\big(x^*,C) \le \gamma \cdot \dist\big(x^*,\partial f(\hat{x})\big)\,.
\]
Hence, together with~\eqref{eq:DL} we get
\[
D_f^{x^*}(x,\hat{x}) \le \tfrac{1}{\sigma_{\min}^+(A)} \cdot L \cdot \gamma \cdot \sqrt{\tfrac{2}{\alpha} \cdot D_f^{x^*}(x,\hat{x})} \cdot \norm[2]{Ax-b} \,,
\]
which implies~\eqref{eq:EB}.
And in case $\tfrac{\alpha}{2} \cdot \varepsilon^2 < D_f^{x^*}(x,\hat{x})$ we apply Lemma~\ref{lem:lineargrowth} to get
\begin{align*}
D_f^{x^*}(x,\hat{x}) & \le c \cdot \left( \sqrt{D_f^{x^*}(x,\hat{x})} + \norm[2]{\hat{x}} +1 \right) \cdot \norm[2]{Ax-b} \\
& \le   c \cdot \sqrt{D_f^{x^*}(x,\hat{x})} \cdot \left(1+\big(\norm[2]{\hat{x}} +1\big) \cdot \sqrt{\tfrac{2}{\alpha}} \cdot \tfrac{1}{\varepsilon} \right) \cdot \norm[2]{Ax-b} \,,
\end{align*}
which also implies~\eqref{eq:EB}.
\qed
\end{proof}

\section{Convergence analysis of the GERK method} \label{sec:convergence}

At first we consider the real case $\KK=\RR$ and prove expected convergence of the generalized extended randomized block Kaczmarz method (GERK) Algorithm~\ref{alg:GERK} to the unique solution of~\eqref{eq:OPgeneral} for suitable strongly convex functions $f$ and $g^*$.
We will derive convergence rates with the help of the following technical lemma.

\begin{lemma} \label{lem:recursion_dk}
Let $a,b>0$, $q \in (0,1)$, and $(d_k)_{k\ge 1}$ be a sequence with $d_k > 0$ and
\begin{equation} \label{eq:recursion_dk}
d_{k+1} \le d_k - a \cdot d_k^2 + b \cdot q^k \,.
\end{equation}
Then there exists some $c>0$ such that $d_k \le \frac{c}{k}$ for all $k\ge 1$.
\end{lemma}

\begin{proof}
To $q \in (0,1)$ we find some $c>0$ such that for all $k \ge 1$ we have
\begin{equation} \label{eq:qk}
\sqrt{\tfrac{2b}{a} \cdot q^k} + b \cdot q^k \le \frac{c}{k+1} \,.
\end{equation}
At first we assume that there are infinitely many indices $k_j$ (in increasing order) for which $d_{k_j}^2 \le \tfrac{2b}{a} \cdot q^{k_j}$.
From~\eqref{eq:recursion_dk} and~\eqref{eq:qk} we infer that for these indices we have $d_{k_j} \le \sqrt{\tfrac{2b}{a} \cdot q^{k_j}} \le \frac{c}{k_j}$ and
\begin{equation} \label{eq:dkj+1}
d_{k_j+1} \le d_{k_j} +  b \cdot q^{k_j} \le \sqrt{\tfrac{2b}{a} \cdot q^{k_j}} + b \cdot q^{k_j}  \le \frac{c}{k_j+1}\,.
\end{equation}
Furthermore, in case $k_{j+1}>k_j+1$, for all $k=k_j+1,\ldots,k_{j+1}-1$ we have $b \cdot q^k < \tfrac{a}{2} \cdot d_k^2$, and thus~\eqref{eq:recursion_dk} yields the recursion
\begin{equation} \label{eq:d_k+1}
d_{k+1} \le d_k - a \cdot d_k^2 + b \cdot q^k \le d_k - \tfrac{a}{2} \cdot d_k^2 \,.
\end{equation}
It follows that $d_{k+1} \le d_k$, and therefore division by $d_k$ and $d_{k+1}$ yields
\[
\frac{1}{d_k} \le \frac{1}{d_{k+1}} - \frac{a}{2} \cdot \frac{d_k}{d_{k+1}} \le \frac{1}{d_{k+1}} - \frac{a}{2} \,,
\]
which together with~\eqref{eq:dkj+1} implies
\[
(k-k_j-1) \cdot \frac{a}{2} \le \sum_{i=k_j+1}^{k-1} \frac{1}{d_{i+1}} - \frac{1}{d_i} = \frac{1}{d_k} - \frac{1}{d_{k_j+1}} \le \frac{1}{d_k} - \frac{k_j+1}{c}\,.
\] 
We conclude that $d_k \le \frac{1}{\min\{\frac{a}{2},\frac{1}{c}\}} \cdot \frac{1}{k}$ for all $k \ge 1$.
In the remaining case that the index set $I = \set{k\in\NN}{d_{k}^2 \le \tfrac{2b}{a} \cdot q^{k}}$ is finite or empty, the assertion follows from inequality~\eqref{eq:d_k+1} for $k\not\in I$ with a similar conclusion.
\qed
\end{proof}

\begin{algorithm}[ht]
  \caption{Generalized Extended Randomized Block Kaczmarz (GERK)}
  \label{alg:GERK}
  \begin{algorithmic}[1]
    \REQUIRE{starting points $x_0=x_{0}^{*}=0\in\RR^n$ and $z_0^*=b\in \RR^m$, $z_0=\nabla g^*(z_{0}^{*})$, matrix $A\in\RR^{m\times n}$ with $M_r$ row-blocks $0 \not=A_{i}\in \RR^{m_i \times n}$ and $N_c$ column-blocks $0 \not=\tilde{A}_{j}\in \RR^{m \times n_j}$ and probabilities $(\tilde p_{j})\in \RR^{N_{c}}$, $(p_i)\in \RR^{M_{r}}$}
    \ENSURE{(approximate) solution of\\$\min_{x \in \RR^n} f(x)$ s.t. $Ax=\hat{y}$, where $\hat{y} = \argmin_{y \in \RR^m} g^*(b-y)$ s.t. $y \in \Rcal(A)$}
    \STATE initialize $k =  0$
    \REPEAT
		\STATE choose a column-block index $j_k=j\in\{1,\dots,N_c\}$ at random with probability $\tilde{p}_j>0$
    \STATE update $z_{k+1}^* = z_{k}^* - \tilde{t}_k \cdot \tilde{A}_{j_{k}}\tilde{A}_{j_{k}}^T z_{k}$ with stepsize $\tilde{t}_k=\frac{1}{L_{g^*} \cdot \norm[2]{\tilde{A}_{j_{k}}}^2}$
		\STATE update $z_{k+1} = \nabla g^*(z_{k+1}^*)$
    \STATE choose a row-block index $i_k=i\in\{1,\dots,M_{r}\}$ at random with probability $p_i>0$
    \STATE update $x_{k+1}^* = x_{k}^* - t_k \cdot A_{i_{k}}^T (A_{i_{k}} x_{k}-b_{i_{k}}+z_{k+1,i_{k}}^*)$ with stepsize $t_k=\frac{1}{L_{f^*} \cdot \norm[2]{A_{i_{k}}}^2}$
    \STATE update $x_{k+1} = \nabla f^*(x_{k+1}^*)$
    \STATE increment $k = k+1$
    \UNTIL{a stopping criterion is satisfied}
  \end{algorithmic}
\end{algorithm}

\begin{theorem} \label{thm:GERK}
Let $g^*:\RR^m \to \RR$ be strongly convex with a Lipschitz-continuous gradient.
Then the iterates $z_k^*$ of the GERK~method from Algorithm~\ref{alg:GERK} converge in expectation to $b-\hat{y}$, where $\hat{y}\in \Rcal(A)$ is the unique solution of
\begin{equation}
\min_{y \in \RR^m} g^*(b-y) \quad \mbox{s.t.}\quad y \in \Rcal(A)\,. \label{eq:DP}
\end{equation}
If the subdifferential mapping of the strongly convex function $f:\RR^n \to \RR$ grows at most linearly,
then the iterates $x_k$ converge in expectation to the corresponding unique solution $\hat{x}$ of
\begin{equation}
\min_{x \in \RR^n} f(x) \quad \mbox{s.t.}\quad Ax=\hat{y}\,. \label{eq:P}
\end{equation}
For some $q \in (0,1)$ and $c>0$ the expected rates of convergence are
\begin{equation} \label{eq:EEzk}
\EE \left[\norm[2]{z_k^*-(b-\hat{y})}^2\right] \le c\cdot q^k  \,, 
\end{equation}
and for all $k \ge 1$
\begin{equation} \label{eq:EExk_sublinear}
\EE \left[\|x_{k}-\hat{x}\|_2^2\right] \le \frac{c}{k} \,.
\end{equation}
Moreover, if a global error bound holds at $\hat{x}$, then we even have
\begin{equation} \label{eq:EExk_linear}
\EE \left[\|x_{k}-\hat{x}\|_2^2\right]  \le c \cdot (1+k) \cdot q^k \,.
\end{equation}
\end{theorem}

\begin{proof}
We split the proof into two parts.
In the first part we show convergence of the iterates $z_k^*$ , and in the second part we show convergence of the iterates $x_k$.\\
\textbf{Part 1}: 
At first we note that the iterates $z_k^*$ are independent from $x_k$ and $i_k$, so that convergence of the $z_k^*$ can be analyzed separately.
In fact, the first part of our method may be reformulated and interpreted as a randomized coordinate descent algorithm~\cite{nesterov2012efficiency} for the problem $\min_{x \in \RR^n} g^*(b-Ax)$, and for the constant stepsizes $\tilde{t}_k$ the linear convergence of the function values follows from Theorem 5.4 under Assumption 2 in~\cite{NC16} together with the error bound in the remark after Theorem 2 in~\cite{ZY13}.
For convenience, and since we also need it for the discussion of non-constant stepsizes in Remark~\ref{rem:stepsize}, here we give a short convergence proof with the help of the results in~\cite{SL19} adapted to the present situation.
The assumptions on $g^*$ imply that the conjugate $g=(g^*)^*$ is also strongly convex with a Lipschitz-continuous gradient.
Hence, this also holds for the objective function $h(z):=g(z)-\scp{b}{z}$ of the dual to~\eqref{eq:DP},
\begin{equation}
\min_{z \in \RR^m} h(z)=g(z)-\scp{b}{z} \quad \mbox{s.t.}\quad A^T z=0 \,.\label{eq:OPg}
\end{equation}
Set $\tilde{z}_k^*:=z_k^*-b$.
Then we have $\tilde{z}_0^*=0$ and $\nabla h^*(\tilde{z}_k^*)=\nabla g^*(\tilde{z}_k^*+b)=\nabla g^*(z_k^*)$.
Hence, the iteration can be written in the form
\[
\tilde{z}_{k+1}^* = \tilde{z}_{k}^* - \tilde{t}_k \cdot \tilde{A}_{j_{k}}\tilde{A}_{j_{k}}^T z_{k} \quad, \quad z_{k+1} = \nabla h^*(\tilde{z}_{k+1}^*)
\]
with initial value $\tilde{z}_0^*=0$.
By Theorem 5.5 in~\cite{SL19} the iterates $z_k$ converge in expectation to the unique solution $\hat{z}$ of~\eqref{eq:OPg} with rate $\EE \left[\|z_{k}-\hat{z}\|_2^2\right]  \le c \cdot q^k$.
By duality and comparison of the optimality conditions of convex programs (cf. Example in~\cite{RW09}), the solution $\hat{y}$ of~\eqref{eq:DP} and the solution $\hat{z}$ of~\eqref{eq:OPg} are related by $\nabla g(\hat{z})=b-\hat{y}$.
Expected convergence of the iterates $z_k^*$ to $b-\hat{y}$ with rate~\eqref{eq:EEzk} then follows from the estimate
\[
\norm[2]{z_k^*-(b-\hat{y})}=\norm[2]{\nabla g(z_k)-\nabla g(\hat{z})} \le L_g \cdot \norm[2]{z_k-\hat{z}} \,.
\]
\textbf{Part 2}: Let $w_k:=A_{i_{k}} x_{k}-b_{i_{k}}+z^*_{k+1,i_{k}}$.
By Definition~\ref{def:D} of the Bregman distance, and since $A_{i_{k}}\hat{x}=\hat{y}_{i_{k}}$, we have
\[
D_f^{x_{k+1}^*}(x_{k+1} ,\hat{x}) = f^*\left(x_{k}^{*} - t_k \cdot A_{i_{k}}^T w_k\right) - \scp{x_{k}^{*}}{\hat{x}} + t_k \cdot \scp{w_k}{\hat{y}_{i_{k}}}  + f(\hat{x})\,.
\]
Using estimate~\eqref{eq:Lip} for $f^*$ yields
\[
D_f^{x_{k+1}^*}(x_{k+1} ,\hat{x}) \le D_f^{x_k^*}(x_k ,\hat{x}) - t_k \cdot \scp{w_k}{A_{i_{k}} x_{k}-\hat{y}_{i_{k}}} + \tfrac{L_{f^*}}{2} \cdot t_k^2 \cdot \norm[2]{A_{i_{k}}^T w_k}^2 \,.
\]
Since $t_k^2=\frac{1}{L_{f^*}^2 \cdot \norm[2]{A_{i_{k}}}^4}$ and $\norm[2]{A_{i_{k}}^T w_k}^2 \le \norm[2]{A_{i_{k}}^T}^2 \cdot\norm[2]{w_k}^2$, we get
\[
D_f^{x_{k+1}^*}(x_{k+1} ,\hat{x}) \le D_f^{x_k^*}(x_k ,\hat{x}) - t_k \cdot \scp{w_k}{A_{i_{k}} x_{k}-\hat{y}_{i_{k}}} + \tfrac{t_k}{2}\cdot \norm[2]{w_k}^2 \,.
\]
We rewrite the last two summands as
\[
\scp{w_k}{A_{i_{k}} x_{k}-\hat{y}_{i_{k}}} = \norm[2]{A_{i_{k}} x_{k}-\hat{y}_{i_{k}}}^2 + \scp{\hat{y}_{i_{k}}-b_{i_{k}}+z^*_{k+1,i_{k}}}{A_{i_{k}} x_{k}-\hat{y}_{i_{k}}}
\]
and 
\begin{align*}
\tfrac{1}{2} \cdot\norm[2]{w_k}^2 &= \tfrac{1}{2} \cdot\norm[2]{A_{i_{k}} x_{k}-\hat{y}_{i_{k}}}^2 + \scp{\hat{y}_{i_{k}}-b_{i_{k}}+z^*_{k+1,i_{k}}}{A_{i_{k}} x_{k}-\hat{y}_{i_{k}}} \\
& \quad + \tfrac{1}{2} \cdot\norm[2]{\hat{y}_{i_{k}}-b_{i_{k}}+z^*_{k+1,i_{k}}}^2
\end{align*}
to get
\begin{align}
\label{eq:RecursionWithoutExpectation_sparse}
D_f^{x_{k+1}^*}(x_{k+1} ,\hat{x})  \le  D_f^{x_k^*}(x_k ,\hat{x}) - \tfrac{t_k}{2} \cdot\norm[2]{A_{i_{k}} x_{k}-\hat{y}_{i_{k}}}^2 +\tfrac{t_k}{2} \cdot \norm[2]{\hat{y}_{i_{k}}-b_{i_{k}}+z^*_{k+1,i_{k}}}^2 \,.
\end{align}
Set $\displaystyle c_1:=\min_{i=1,\ldots,M_r} \tfrac{p_i}{2 \cdot L_{f^*} \cdot \norm[2]{A_i}^2}$ and $\displaystyle  c_2:=\max_{i=1,\ldots,M_r} \tfrac{p_i}{2 \cdot L_{f^*} \cdot \norm[2]{A_i}^2}$.
Then we have $0<c_1 \le p_i \cdot \tfrac{t_k}{2} \le c_2$ for all $i=1,\ldots,M_r$. Averaging \eqref{eq:RecursionWithoutExpectation_sparse} over the random variables $i_0,j_0,...,i_{k-1},j_{k-1}$ and using linearity of the expectation, we obtain the recursion
\[
\EE \left[D_f^{x_{k+1}^*}(x_{k+1} ,\hat{x})\right]  \le \EE \left[D_f^{x_{k}^*}(x_{k} ,\hat{x})\right]  - c_1 \cdot \EE \left[\norm[2]{A x_k-\hat{y}}^2\right]+ c_2 \cdot \EE \left[\norm[2]{\hat{y}-b+z^*_{k+1}}^2\right]\,.
\]
Using~\eqref{eq:EEzk} we arrive at
\begin{equation}
\EE \left[D_f^{x_{k+1}^*}(x_{k+1} ,\hat{x})\right] \le \EE \left[D_f^{x_{k}^*}(x_{k} ,\hat{x})\right] - c_1 \cdot \EE \left[\norm[2]{A x_k-\hat{y}}^2\right] + c_2 \cdot c \cdot q^{k+1}\,. \label{eq:EErec}
\end{equation}
This recursion implies boundedness of $\EE \left[D_f^{x_{k}^*}(x_{k} ,\hat{x})\right]$, because by the choice $x_0=x_{0}^{*}=0$ the initial Bregman distance $D_f^{x_0^*}(x_0,\hat{x}) = f(\hat{x})$ is finite.
For ease of notation, in the following we use a generic constant $c>0$ that is independent of the iteration index $k$ and the random choices of the algorithm. 
By Lemma~\ref{lem:lineargrowth}, the linear growth assumption on $\partial f$ implies
\begin{align*}
\EE \left[D_f^{x_{k}^*}(x_{k} ,\hat{x})\right]  &\le c \cdot \EE \left[\sqrt{D_f^{x_{k}^*}(x_{k} ,\hat{x})} \cdot \norm[2]{Ax_k-\hat{y}}\right] + c \cdot \EE \left[\norm[2]{Ax_k-\hat{y}} \right] \\
 &\le  c \cdot \sqrt{\EE \left[D_f^{x_{k}^*}(x_{k} ,\hat{x})\right]} \cdot \sqrt{\EE \left[\norm[2]{Ax_k-\hat{y}}^2\right]} + c \cdot \EE \left[\norm[2]{Ax_k-\hat{y}} \right] \\
&\le   c \cdot \sqrt{\EE \left[\norm[2]{Ax_k-\hat{y}}^2\right]}\,,
\end{align*}
which yields
\[
\left(\EE \left[D_f^{x_{k}^*}(x_{k} ,\hat{x})\right]\right)^2 \le c \cdot \EE \left[\norm[2]{Ax_k-\hat{y}}^2\right]\,.
\]
We insert this inequality into recursion~\eqref{eq:EErec} to get
\[
\EE \left[D_f^{x_{k+1}^*}(x_{k+1} ,\hat{x})\right] \le \EE \left[D_f^{x_{k}^*}(x_{k} ,\hat{x})\right] - c \cdot \left(\EE \left[D_f^{x_{k}^*}(x_{k} ,\hat{x})\right]\right)^2 + c \cdot q^{k+1}\,.
\]
The sublinear convergence rate~\eqref{eq:EExk_sublinear} then follows from Lemma~\ref{lem:recursion_dk}.
Now we turn to the asymptotically better rate~\eqref{eq:EExk_linear} under the stronger assumption that a global error bound of the form
\[
D_f^{x_{k}^*}(x_{k} ,\hat{x}) \le \gamma \cdot \norm[2]{A x_k-\hat{y}}^2
\]
holds with some constant $\gamma>0$.
We set $q_1:=\max\{0,1-c_1/ \gamma\}$.
Then we have $q_1 \in [0,1)$, and inserting the error bound into~\eqref{eq:EErec} we get
\[
\EE \left[D_f^{x_{k+1}^*}(x_{k+1} ,\hat{x})\right] \le q_1 \cdot \EE \left[D_f^{x_{k}^*}(x_{k} ,\hat{x})\right] + c_2 \cdot c \cdot q^{k+1}\,.
\]
Finally, we set $\tilde{q}:=\max\{q_1,q\}$ and conclude inductively that
\[
\EE \left[D_f^{x_{k}^*}(x_{k} ,\hat{x})\right] \le c \cdot \tilde{q}^k + c \cdot k \cdot \tilde{q}^k \,,
\]
from which the rate~\eqref{eq:EExk_linear} follows by~\eqref{eq:D}.
\qed
\end{proof}

\begin{remark} \label{rem:stepsize}
According to~\cite{SL19}, the stepsize $\tilde{t}_k$ for the $z^*_k$-update in line 4 of Algorithm~\ref{alg:GERK} may also be chosen as
\[
\tilde{t}_k=\tfrac{1}{L_{g^*}} \cdot \tfrac{\norm[2]{\tilde{A}_{j_k}^T z_k}^2}{\norm[2]{\tilde{A}_{j_k} \tilde{A}_{j_k}^T z_k}^2}
\]
or determined by an exact linesearch.
But so far we do not know whether we can also choose the stepsize $t_k$ for the $x^*_k$-update in line 7 by an exact linesearch or as $t_k=\frac{1}{L_{f^*}} \cdot \frac{\norm[2]{w_k}^2}{\norm[2]{A_{i_{k}}^T w_k}^2}$ with $w_k:=A_{i_{k}} x_{k}-b_{i_{k}}+z^*_{k+1,i_{k}}$.
The main problem with this choice here seems to be that we only have a lower estimate $t_k \ge \frac{1}{L_{f^*} \cdot \norm[2]{A_{i_{k}}}^2}$, but after inequality~\eqref{eq:RecursionWithoutExpectation_sparse} in the above proof we would also need a suitable upper estimate (note that $w_k$ need not be contained in $\Rcal(A_{i_{k}})$).
\end{remark}

To apply Theorem~\ref{thm:GERK} in the complex case $\KK=\CC$, we just split the variables into real and imaginary parts.
In this way, a complex linear system $Ax=b$ can equivalently be written as a real linear system of the form
\[
\begin{pmatrix}
\Re(A) &, -\Im(A) \\
\Im(A) &, \Re(A)
\end{pmatrix}
\cdot
\begin{pmatrix}
\Re(x) \\ \Im(x)
\end{pmatrix}
=
\begin{pmatrix}
\Re(b) \\ \Im(b)
\end{pmatrix}
\]
and a vector update as in lines 4 and 7 of Algorithm~\ref{alg:GERK} for a complex vector then corresponds to block updates of the real and imaginary parts.
But we must take some care when we consider a function $f:\CC^n\to \RR$ in complex variables as a function $f:\RR^{2n} \to \RR$ in real variables.
In particular, there is a notable subtlety regarding the sparsity promoting function $f(x) = \lambda \cdot \norm[1]{x} + \tfrac{1}{2} \cdot \norm[2]{x}^{2}$ for complex vectors $x \in \CC^n$.
Considering it as a real function of the form
\[
f\big(\Re(x),\Im(x)\big)=\lambda \cdot \big(\norm[1]{\Re(x)} +\norm[1]{\Im(x)}\big) + \tfrac{1}{2} \cdot \big(\norm[2]{\Re(x)}^{2} +\norm[2]{\Im(x)}^{2}\big) \,,
\]
the gradient $\nabla f^*$ of the conjugate function would just be componentwise shrinkage of the vector $\big(\Re(x),\Im(x)\big)$, i.e. sparsity of the real and imaginary part is enforced seperately.
On the one hand, this means that sparsity of the real vector $\big(\Re(x),\Im(x)\big)$ does not necessarily imply sparsity of the complex vector $x$.
On the other hand, a global error bound is guaranteed to hold, cf.~Examples~\ref{exmp:calm} (a),~\ref{exmp:linreg} (a), and~\ref{exmp:LinearGrowth}.
A more suitable way to enforce sparsity of a complex vector seems to be to just use the complex $\ell_1$-norm, i.e.
\[
f\big(\Re(x),\Im(x)\big)=\lambda \cdot \sum_{j=1}^n \sqrt{\big(\Re(x_j)\big)^2+\big(\Im(x_j)\big)^2} + \tfrac{1}{2} \cdot \big(\norm[2]{\Re(x)}^{2} +\norm[2]{\Im(x)}^{2}\big) \,,
\]
where, by~Examples~\ref{exmp:calm} (c) and (d), the gradient $\nabla f^*$ of the conjugate function amounts to componentwise shrinkage of the complex vector $x$,
\begin{equation} \label{eq:shrinkage_complex}
\Big(\:\big(\nabla f^*(x)\big)_j \: \hat{=}\:\Big) \quad \quad \big(S_{\lambda}(x)\big)_j = \max\{|x_j|-\lambda,0\} \cdot \tfrac{x_j}{|x_j|}  \quad , x \in \CC^n \,,
\end{equation}
i.e. sparsity of the real and imaginary part is enforced simultaneously.
But since this is a special form of group sparsity, we can guarantee a global error bound, and hence the better rate~\eqref{eq:EExk_linear}, only under an additional regularity assumption as in Example~\ref{exmp:linreg} (b).
\begin{remark}
\label{rem:ComplexGREKComplexOperations}
Algorithm~\ref{alg:GERK} can also be directly implemented with complex number operations.
We just have to replace the transposed matrices $\tilde{A}_{j_{k}}^T$ and $A_{i_{k}}^T$ in lines 4 and 7 by the complex adjoints $\overline{\tilde{A}}_{j_{k}}^T$ and $\overline{A}_{i_{k}}^T$, respectively.
The updates in lines 5 and 8 must be performed by replacing the real gradient mappings $\nabla g^*$ and $\nabla f^*$ with the corresponding complex operators, e.g. using the complex shrinkage operator~\eqref{eq:shrinkage_complex}, cf. \cite{Bur16, Sar20}.
Note that the expressions for the Huber function and its gradient in Example~\ref{exmp:Huber} are also meaningful for complex vectors $x$, and the corresponding real function is still strongly convex and has a Lipschitz-continuous gradient.
\end{remark}

\begin{example}
\label{ex:Examples_g_star}
Here are some concrete choices for the functions $f$ and $g^*$ that can be used in both the real and complex case $\KK=\RR$ or $\KK=\CC$, so that the assumptions in Theorem~\ref{thm:GERK} are fulfilled.
We indicate by (RA) if a regularity assumption as in Example~\ref{exmp:linreg} (b) is needed for $f$ to ensure a global error bound and hence the better rate~\eqref{eq:EExk_linear}.
\begin{enumerate}
    \item \textbf{(Least squares)} $g^*(y)=\tfrac{1}{2} \cdot \norm[2]{y}^2$
    \item \textbf{(Impulsive noise)} $g^*(y) = r_{\varepsilon}(y) + \tfrac{\tau}{2} \cdot \norm[2]{y}^2$ with the Huber function $r_{\varepsilon}$
    \item \textbf{(Minimum $2$-norm)} $f(x)=\tfrac{1}{2} \cdot \norm[2]{x}^2$
    \item \textbf{(Sparsity, (RA) needed only for $\KK=\CC$)} $f(x) = \lambda \cdot \norm[1]{x} + \tfrac{1}{2} \cdot \norm[2]{x}^{2}$ 
    \item \textbf{(Group sparsity (RA))} $f(x) = \lambda \cdot \sum_{j=1}^K \norm[2]{x_j} + \tfrac{1}{2} \cdot \norm[2]{x}^{2}$
    \item \textbf{(Low rank matrices (RA))} $f(X)=\lambda \norm[*]{X} + \tfrac{1}{2}\norm[F]{X}^2$
\end{enumerate}
\end{example}

Note that instead of $f(x) = \lambda \cdot \norm[1]{x} + \tfrac{1}{2} \cdot \norm[2]{x}^{2}$ we could also use $f(x) = r_{\varepsilon}(x) + \tfrac{\tau}{2} \cdot \norm[2]{x}^2$ as sparsity promoting function, as was done in~\cite{P15}.
But this requires tuning the two parameters $\varepsilon$,$\tau$ instead of only $\lambda$.
On the contrary, so far we could not prove convergence for non-smooth data misfit functions $g^*$, so that we cannot use $g^*(y) = \lambda \cdot \norm[1]{y} + \tfrac{1}{2} \cdot \norm[2]{y}^{2}$ for impulsive noise.

\section{Numerical examples} \label{sec:NumericalExamples}
In this part, we report numerical results of Algorithm \ref{alg:GERK} (GERK) for multiple settings and compare with the Sparse Randomized Kaczmarz method~(SRK) from~\cite{LWSM14} and the Extended Randomized Kaczmarz method~(REK) from~\cite{ZF13}. All examples are run in \texttt{MATLAB 2019b} on a computer with an Intel Core i7 processor with 4 cores at 1,2 GHz and 16 GB RAM. 

We consider two kinds of experiments, one to find sparse least squares solutions (which assume normally distributed noise) and one in which we aim to find sparse solutions to inconsistent systems under impulsive noise. 
\begin{enumerate}
\item[(i)] \textbf{Least squares solutions:} In a first experiment, we want to find sparse solutions of the least squares problem $\min\|Ax-b\|_2$ for the real and complex case. We use the functions $g^*(y) = \tfrac12\norm{y}^{2}$ from Example \ref{ex:Examples_g_star}(a) and $f(x) = \lambda\norm[1]{x}+\tfrac12\norm[2]{x}^{2}$ from Example \ref{ex:Examples_g_star}(d) and hence, we refer to the resulting method as GERK-(a,d).

  Similar to~\cite{Du19} we fix some dimensions $m$ and $n$ and some rank $r<\min(m,n)$ and construct $A\in\RR^{m\times n}$ (or $\CC^{m\times n}$, respectively) as follows. For two matrices $U,V$ with orthonormal columns (generated with the MATLAB command \texttt{orth}), we set $A=U\Sigma V^T$ with a diagonal matrix $\Sigma$ which $r$ nonzero entries on the diagonal which were sampled from the uniform distribution on $[0.001, 100]$. Then we construct a sparse $\hat x$ with normally distributed non-zero entries and set
  \begin{align*}
\hat b = A\hat{x}, \quad b = \hat b + \eta_{\mathcal R(A)^\perp}
  \end{align*}
  with noise $\eta_{\mathcal R(A)^\perp} = Nv \in \mathcal{R}(A)^\perp = \mathcal{N}(A^*)$, where the columns of $N$ form an orthonormal basis of $\mathcal N(A^*)$ and $v$ is a random vector uniformly distributed on a sphere $\partial\mathcal B_\rho(0)$ with radius $\rho=\alpha \|\hat b\|_2$ and a factor $\alpha$. The factor $\alpha$ is exactly the relative noise level.

Note that we do not consider a matrix with full rank due to the following reason:
If $m\leq n$ and $A$ has full rank, it holds $\mathcal{R}(A) = \RR^m$. If $m\geq n$, the matrix $A^TA$ is invertible and  the least-squares solution of the possibly inconsistent system $Ax=b$ is unique. In both cases, the sparse solution can be found either by the existing randomized sparse Kaczmarz method or by the existing randomized extended Kaczmarz method. \\

Figure~\ref{fig:noise_RAc_real_err} and Figure~\ref{fig:noise_RAc_real_x_components} show the results for the real case for 50 runs with $m=1000$, $n=500$, rank $r = 250$, sparsity $s = 25$, noise level $\alpha=5$, penalty $\lambda=5$ and uniform probabilities $p$ and $\tilde p$. In Table~\ref{tab:noise_RAc} we report the sparsity of the last iterates (min, median and max). One observes that the randomized sparse Kaczmarz method (blue) does not find any least squares  solution and in fact, the method does not even converge (because the system is inconsistent). The randomized extended Kaczmarz method (black) does indeed find a least squares solution, but fails to find the sparse one. The GERK-(a,d) method does converge to a sparse least squares solution and indeed recovers $\hat x$. Moreover, GERK-(a,d) is even faster than the standard randomized extended Kaczmarz method. In Table~\ref{tab:noise_RAc} we note that the last iterates of the sparse randomized Kaczmarz method are not as sparse as they should be, while GERK-(a,d) is able to produce sparser solutions. Figure~\ref{fig:noise_RAc_complex_err} and Figure~\ref{fig:noise_RAc_complex_x_components} report the result for the complex case in the same setup and we can draw the same conclusion.

\begin{figure}[htb]
  \centering 
  \includegraphics{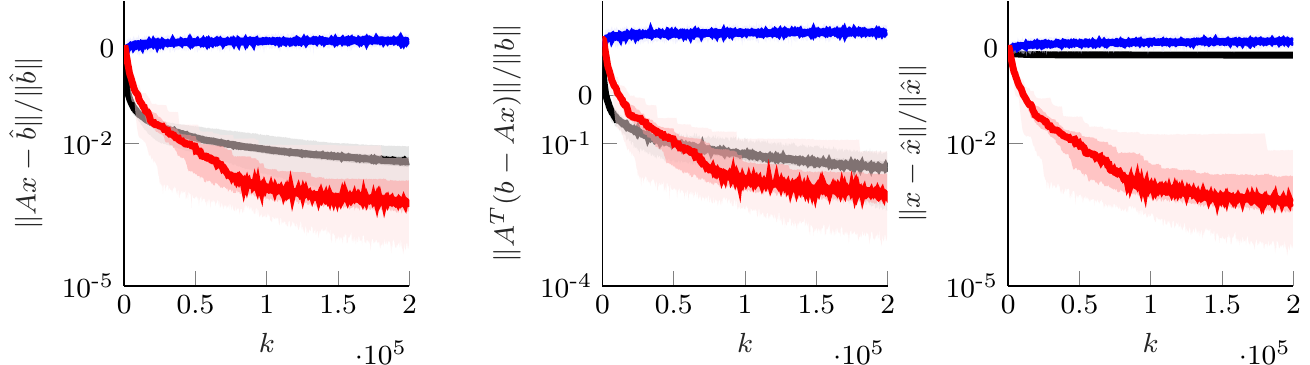}
  \caption{
  A comparison of real randomized extended Kaczmarz (black), randomized sparse Kaczmarz (blue) and GERK(a,d) method (red). Experiment (i) with $m = 1000, n = 500,$ sparsity$=25$, rank $r=250$, $\alpha=5$, $\lambda=5$, uniform probabilities $p$, $\tilde p$ and 50 repeats. 
  Left: Plot of relative residual $\|Ax-\hat b\|_2/\|\hat b\|_2$, middle: Plot of relative gradient norm $\|A^T(b-Ax)\|_2/\|b\|_2$, right: plot of relative distance $\|x-\hat x\|_2/\|\hat x\|_2$ to the initial sparse solution $\hat x$. Thick line shows median over all trials, light area is between min and max, darker area indicates 25th and 75th quantile}
  \label{fig:noise_RAc_real_err}
\end{figure}

\begin{figure}[htb]
  \centering
  \includegraphics{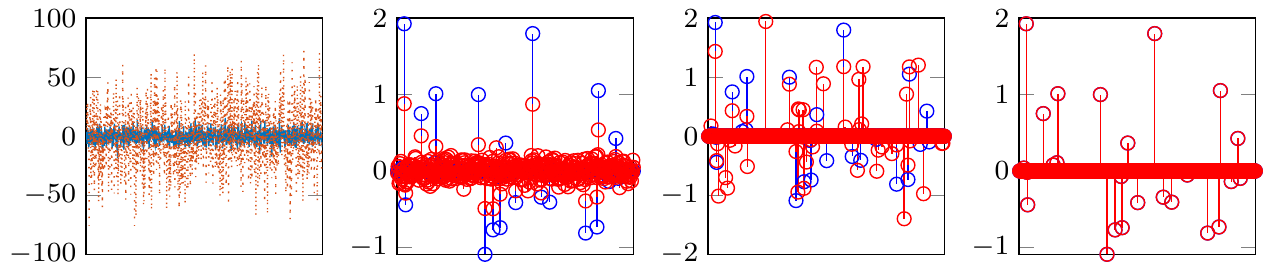}
  \caption{Approximated solution in the experiment from Figure \ref{fig:noise_RAc_real_err}. Left: Plot of $\hat b$ (blue) and noisy $b$ (red), right: Plot of $\hat x$ (blue) and last iterate $x$ (red) of randomized extended Kaczmarz, randomized sparse Kaczmarz and GERK-(a,d) method}
  \label{fig:noise_RAc_real_x_components}
\end{figure}

\begin{figure}[htb]
  \centering
  \includegraphics{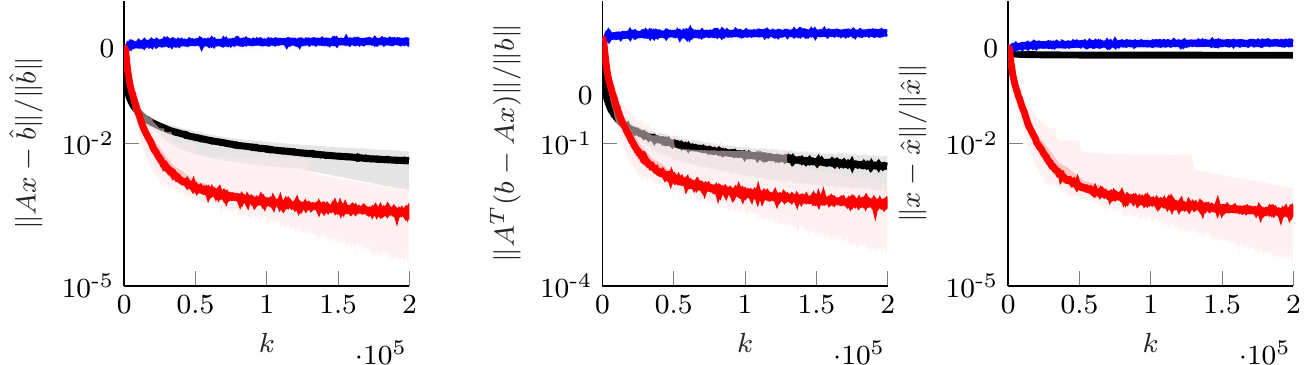}
  \caption{
  Experiment (i) with complex $A$, $b$ and $\hat x$ and the complex method, cf. Remark \ref{rem:ComplexGREKComplexOperations}, with parameters as in Figure \ref{fig:noise_RAc_real_err}.}
  \label{fig:noise_RAc_complex_err}
\end{figure}

\begin{figure}[htb]
  \centering
  \includegraphics{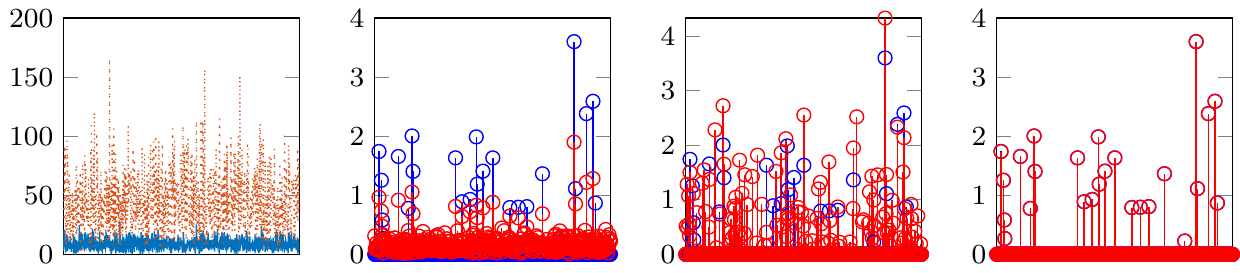}
    \caption{Approximated solution in the experiment from Figure \ref{fig:noise_RAc_complex_err}, only absolute values. Left: Plot of $\hat b$ (blue) and noisy $b$ (red), right: Plot of $\hat x$ (blue) and last iterate $x$ (red) of randomized extended Kaczmarz, randomized sparse Kaczmarz and GERK-(a,d) method}
    \label{fig:noise_RAc_complex_x_components}
\end{figure}

\begin{table}[htb]
    \centering
    \begin{tabular}{ccc}
    \toprule
        Algorithm & Figure \ref{fig:noise_RAc_real_x_components} & Figure \ref{fig:noise_RAc_complex_x_components} \\ \midrule
        REK & 499/500/500 & 500/500/500 \\ 
        SRK & 56/75.5/99 & 92/123/159 \\ 
        GERK-(a,d) & 25/27/42 & 25/26/31 \\ \bottomrule
    \end{tabular}
    \caption{Sparsity of last iterates (\#$|x_{N,i}|>10^{-5}$) in Figures \ref{fig:noise_RAc_real_x_components} and \ref{fig:noise_RAc_complex_x_components} (min/median/max)}\label{tab:noise_RAc}
\end{table}

Note that we did not report results on noise in the range of $A$. The method still works in this case and converges to some approximate least squares solution with error in the order of the level of the noise in $\mathcal R(A)$.

\item[(ii)] \textbf{Impulsive noise:} In a second experiment, we use mainly the same setup as in the the first experiment, but instead of noise in $\mathcal R(A)^\perp$ we add impulsive noise and use Algorithm \ref{alg:GERK} with $g^*(y) = r_{\varepsilon}(y) + \tfrac\tau2\norm[2]{y}^{2}$ from Example~\ref{ex:Examples_g_star}(b) and $f$ from Example \ref{ex:Examples_g_star}(d) (as in the first experiment). More concretely, after choosing a sparse solution $\hat x$ as in (i) we set 
\begin{align*}
    \hat b = A\hat{x}, \quad b = \hat b + \eta_{\mathrm{impulsive}},
\end{align*} 
where we generate $\eta_{\mathrm{impulsive}}$ by choosing a random subset $I\subset\{1,...,n\}$ with $\tilde n = \lceil n/20 \rceil$ many elements and setting 
\begin{align*}
    \big(\eta_{\mathrm{impulsive}}\big)_i = \begin{cases}
    s_i\cdot \alpha \cdot \norm[\infty]{\hat b}, & i\in I, \\
    0, & \text{otherwise}
    \end{cases}
\end{align*}
with random signs $s_i\in\{-1,1\}$. In our experiments, we have set $\alpha = 5.$
For the matrix $A$ we choose singular values in $[0.001,10]$ and use $\varepsilon = 10^{-2}$ and $\tau=10^{-3}$, $m=1000$, $n=500$, $\lambda=10$, sparsity $25$ and rank $250$.
The results of 50 trials of the real case are reported in Figure~\ref{fig:impulsive_noise_real_err} and Figure~\ref{fig:impulsive_noise_real_x_components}. In Figure~\ref{fig:impulsive_noise_complex_err} and Figure~\ref{fig:impulsive_noise_complex_x_components} we report the results of the complex case. Here, we set \begin{align*}
    \big(\eta_{\mathrm{impulsive}}\big)_j = \begin{cases}
    \frac{1}{\sqrt{2}} \cdot (s_j + it_j)\cdot \alpha \cdot \norm[\infty]{\hat b}, & j\in I, \\
    0 & \text{otherwise}
    \end{cases}
\end{align*} 
with random signs $s_j,t_j\in\{-1,1\}$.

We observe that, different to the GERK-(a,d) method, the GERK-(b,d) method is able to reconstruct the sparse vectors and gives the sparsest iterates. The exemplary plot of the vector components suggests that, when applying the GERK-(b,d) method, the remaining nonzero components might even vanish after more steps. Again, the sparsity of the last iterates is reported in Table~\ref{tab:impulsive_noise}. The middle figures in Figure \ref{fig:impulsive_noise_real_err} and Figure \ref{fig:impulsive_noise_complex_err} show the gradients of the lower level objective in \eqref{eq:DP} w.r.t. $x$. We observe that the REK and GERK-(a,d) method indeed solve \eqref{eq:DP} with $g$ from Example 4.5(a), and the GERK-(b,d) method seems to converge to a solution of \eqref{eq:DP} with $g$ from Example 4.5(b). The relative residuals, the relative gradient norm and the relative error for the REK method behaved like those for the GERK-(a,d) method (i.e., the black curves in the respective plots in Figure~\ref{fig:impulsive_noise_real_err} and Figure~\ref{fig:impulsive_noise_complex_err} are hidden behind the blue and red curve). We would also like to mention that the speed of the reconstruction depends heavily on the condition of the matrix~$A$. For~$A$ with condition as in our experiment and if $b$ is corrupted by small additional noise, the method is still able to approximate the sparse solution with error in the order of the level of the additional noise.

\begin{figure}[htb]
  \centering
  \includegraphics{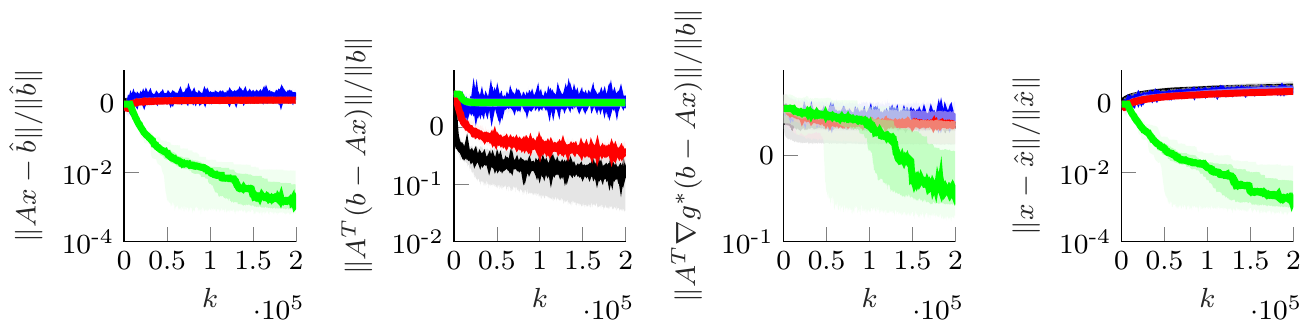}
  \caption{
  A comparison of randomized sparse Kaczmarz (blue), randomized extended Kaczmarz (black), GERK-(a,d) (red) and GERK-(b,d) method (green). Experiment (ii) with $m = 1000, n = 500,$ sparsity$=25$, rank $r=250$, $\epsilon=10^{-2}$, $\tau=10^{-3}$, $\lambda=10$, uniform probabilities $p$, $50$ repeats. 
  From left to right: Plot of relative residual $\|Ax-\hat b\|_2/\|\hat b\|_2$, relative gradient norm $\|A^T(b-Ax)\|_2/\|b\|_2$, $\|A^T\nabla g^*(b-Ax)\|_2/\|b\|_2$ and relative error $\|x-\hat x\|_2/\|\hat x\|_2$. Thick line shows median over all trials, light area is between min and max, darker area indicates 25th and 75th quantile.}
  \label{fig:impulsive_noise_real_err}
\end{figure}

\begin{figure}[htb]
  \centering
  \includegraphics{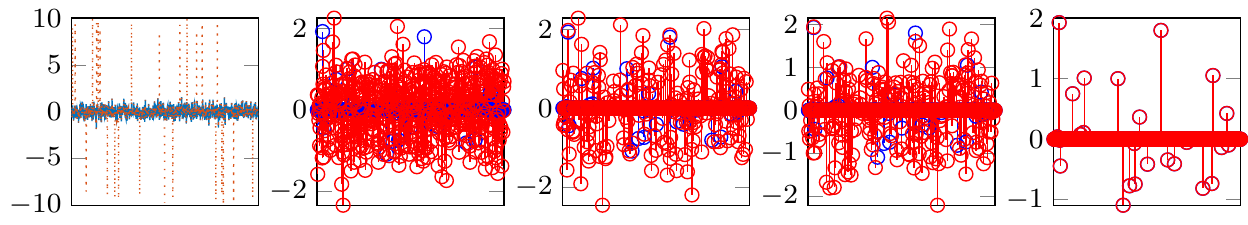}
      \caption{
Approximated solution in the experiment from Figure \ref{fig:impulsive_noise_real_err}. Left: Plot of $\hat b$ (blue) and noisy $b$ (red), right: Plot of $\hat x$ (blue) and last iterate $x$ (red) of randomized extended Kaczmarz, randomized sparse Kaczmarz, GERK-(a,d) and GERK-(b,d) method
    }
    \label{fig:impulsive_noise_real_x_components}
\end{figure}

\begin{figure}[htb]
  \centering
  \includegraphics{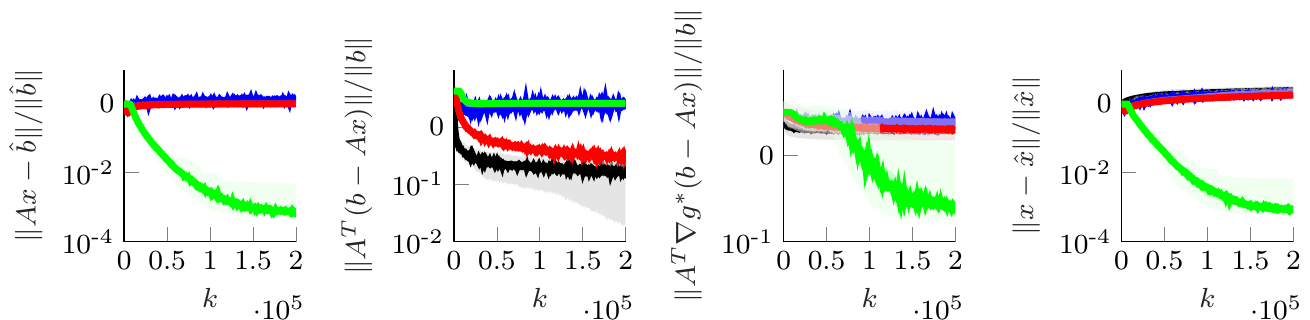}
  \caption{
  Experiment (ii) with complex $A$, $b$ and $\hat x$ and the complex method, cf. Remark \ref{rem:ComplexGREKComplexOperations}, with parameters as in Figure \ref{fig:impulsive_noise_real_err}, and with independent uniform noise in real and imaginary part.}
  \label{fig:impulsive_noise_complex_err}
\end{figure}

\begin{figure}[htb]
  \centering
  \includegraphics{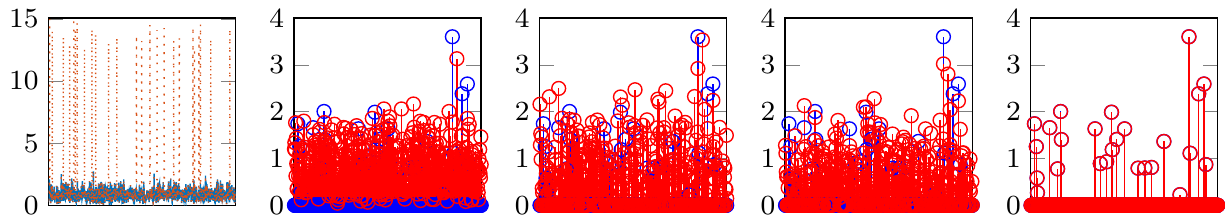}
    \caption{
Approximated solution in the experiment from Figure \ref{fig:impulsive_noise_complex_err}, only absolute values. Left: Plot of $\hat b$ (blue) and noisy $b$ (red), right: Plot of $\hat x$ (blue) and last iterate $x$ (red) of randomized extended Kaczmarz, randomized sparse Kaczmarz, GERK-(a,d) and the GERK-(b,d) method
    }
    \label{fig:impulsive_noise_complex_x_components}
\end{figure}

\begin{table}[htb]
    \centering
    \begin{tabular}{lcc}
    \toprule  
        Algorithm & Figure \ref{fig:impulsive_noise_real_x_components} & Figure \ref{fig:impulsive_noise_complex_x_components} \\ \midrule
        REK & 499/500/500 & 500/500/500  \\ 
        SRK & 154/182.5/222 & 253/287.5/326  \\ 
        GERK-(a,d) & 197/219/256 & 307/329/352  \\
        GERK-(b,d) & 28/36/51 & 30/45/70  \\
        \bottomrule
    \end{tabular}
    \caption{Sparsity of last iterates (\#$|x_{N,i}|>10^{-5}$) in Figures \ref{fig:impulsive_noise_real_x_components} and \ref{fig:impulsive_noise_complex_x_components}}\label{tab:impulsive_noise}
\end{table}
\end{enumerate}

\section{Conclusion}
\label{sec:conclusion}

We showed that the extended randomized Kaczmarz method can be further generalized to the case of sparse least squares solutions for inconsistent systems. We can even allow different smooth and strongly convex data misfit functions $g^{*}$ that can be modeled to cover different noise models.  Moreover, under a global error bound we obtain linear convergence in this case and we show that these global error bounds hold under mild regularity assumptions. Our numerical experiments indicate that this generalization is indeed successful for rank deficient inconsistent least squares problems for both large normally distributed noise in the complement of the range of the system matrix and the case of large impulsive noise. Future research could consider adjoint mismatch as in~\cite{LRS18}, or averaging as in~\cite{moorman2021randomized}.
Furthermore, although our analysis includes block variants (which was important for the complex case $\KK=\CC$), here we did not concentrate on proving or validating numerically an actual advantage of using block variants.
This is an important topic for future research, and it would be interesting to know whether the results in~\cite{Wu22} can be generalized to our setting.

\bibliography{./literature}
\bibliographystyle{plain}

\end{document}